\documentclass[10pt,righttag]{amsart}
\usepackage{amssymb,verbatim,amscd,mathrsfs,stmaryrd,wasysym,latexsym, bbm}
\usepackage[all]{xy} 
\usepackage{cancel,color}
\usepackage{lscape}
\begin{document}

\newcommand{\V}{{\mathcal V}}      
\renewcommand{\O}{{\mathcal O}}
\newcommand{\LL}{\mathcal L}
\newcommand{\mcO}{\mathcal{O}}
\newcommand{\mcF}{\mathcal{F}}
\newcommand{\mcL}{\mathcal{L}}
\newcommand{\mcG}{\mathcal{G}}
\newcommand{\mcH}{\mathcal{H}}
\newcommand{\mcI}{\mathcal{I}}
\newcommand{\mcK}{\mathcal{K}}

\newcommand{\Ext}{\hbox{\rm Ext}}
\newcommand{\Tor}{\hbox{\rm Tor}}
\newcommand{\Hom}{\hbox{Hom}}
\newcommand{\Spec}{\hbox{Spec }}
\newcommand{\Proj}{\hbox{Proj }}
\newcommand{\Mod}{\hbox{Mod}}
\newcommand{\GrMod}{\hbox{GrMod}}
\newcommand{\grmod}{\hbox{gr-mod}}
\newcommand{\Tors}{\hbox{Tors}}
\newcommand{\gr}{\hbox{gr}}
\newcommand{\tors}{\hbox{tors}}
\newcommand{\rank}{\hbox{rank}}
\newcommand{\End}{\hbox{{\rm End}}}
\newcommand{\Der}{\hbox{Der}}
\newcommand{\GKdim}{\hbox{GKdim}}
\newcommand{\gldim}{\hbox{gldim}}
\newcommand{\im}{\hbox{im}}
\renewcommand{\ker}{\hbox{ker }}
\newcommand{\coker}{\hbox{coker }}
\newcommand{\Char}{\hbox{char}}
\newcommand{\colim}{\hbox{colim}}
\newcommand{\depth}{\hbox{depth}}
\def\bee{\begin{eqnarray}}
\def\eee{\end{eqnarray}}

\newcommand{\AGr}{\hbox{A-Gr}}
\newcommand{\cohmod}{\hbox{Coh-Mod}}
\newcommand{\finpres}{\hbox{Fin-Pres}}
\newcommand{\fdim}{\hbox{Fin-Dim}}
\newcommand{\qgr}{\hbox{qgr}}

\newcommand{\lonto}{{\protect \longrightarrow\!\!\!\!\!\!\!\!\longrightarrow}}

\renewcommand{\c}{\cancel}
\newcommand{\fh}{\frak h}
\newcommand{\fp}{\frak p}
\newcommand{\fq}{\frak q}
\newcommand{\fr}{\frak r}
\newcommand{\mf}{\mathfrak}
\newcommand{\m}{{\mu}}
\newcommand{\gl}{{\frak g}{\frak l}}
\newcommand{\ssl}{{\frak s}{\frak l}}
\newcommand{\tw}{{\rm tw}}

\newcommand{\ds}{\displaystyle}
\newcommand{\s}{\sigma}
\renewcommand{\l}{\lambda}
\renewcommand{\a}{\alpha}
\renewcommand{\b}{\beta}
\newcommand{\G}{\Gamma}
\newcommand{\g}{\gamma}
\newcommand{\z}{\zeta}
\newcommand{\e}{\epsilon}
\renewcommand{\d}{\delta}
\newcommand{\p}{\rho}
\renewcommand{\t}{\tau}
\newcommand{\n}{\nu}
\newcommand{\x}{\chi}
\newcommand{\w}{\omega}
\renewcommand{\i}{\iota}

\newcommand{\A}{{\Bbb A}}
\newcommand{\C}{{\Bbb C}}
\newcommand{\N}{{\Bbb N}}
\newcommand{\Z}{{\Bbb Z}}
\newcommand{\ZZ}{{\Bbb Z}}
\newcommand{\Q}{{\Bbb Q}}
\renewcommand{\k}{\mathbb K}

\newcommand{\E}{{\mathcal E}}
\newcommand{\K}{{\mathcal K}}
\renewcommand{\S}{{\mathcal S}}
\newcommand{\T}{{\mathcal T}}

\newcommand{\GL}{{GL}}

\newcommand{\rowxy}{(x\ y)}
\newcommand{\colxy}{ \left({\begin{array}{c} x \\ y \end{array}}\right)}
\newcommand{\scolxy}{\left({\begin{smallmatrix} x \\ y
\end{smallmatrix}}\right)}

\renewcommand{\P}{{\Bbb P}}

\newcommand{\la}{\langle}
\newcommand{\ra}{\rangle}
\newcommand{\tensor}{\otimes}
\newcommand{\tsr}{\tensor}
\newcommand{\ol}{\overline}

\newtheorem{thm}{Theorem}[section]
\newtheorem{lemma}[thm]{Lemma}
\newtheorem{cor}[thm]{Corollary}
\newtheorem{prop}[thm]{Proposition}
\newtheorem{claim}[thm]{Claim}

\theoremstyle{definition}
\newtheorem{defn}[thm]{Definition}
\newtheorem{notn}[thm]{Notation}
\newtheorem{ex}[thm]{Example}
\newtheorem{rmk}[thm]{Remark}
\newtheorem{rmks}[thm]{Remarks}
\newtheorem{note}[thm]{Note}
\newtheorem{example}[thm]{Example}
\newtheorem{problem}[thm]{Problem}
\newtheorem{ques}[thm]{Question}
\newtheorem{conj}[thm]{Conjecture}
\newtheorem{thingy}[thm]{}

\newcommand{\onto}{{\protect \rightarrow\!\!\!\!\!\rightarrow}}
\newcommand{\donto}{\put(0,-2){$|$}\put(-1.3,-12){$\downarrow$}{\put(-1.3,-14.5) 

{$\downarrow$}}}

\newcounter{letter}
\renewcommand{\theletter}{\rom{(}\alph{letter}\rom{)}}

\newenvironment{lcase}{\begin{list}{~~~~\theletter} {\usecounter{letter}
\setlength{\labelwidth4ex}{\leftmargin6ex}}}{\end{list}}

\newcounter{rnum}
\renewcommand{\thernum}{\rom{(}\roman{rnum}\rom{)}}

\newenvironment{lnum}{\begin{list}{~~~~\thernum}{\usecounter{rnum}
\setlength{\labelwidth4ex}{\leftmargin6ex}}}{\end{list}}

\thispagestyle{empty}

\title{Graded coherence of certain extensions of graded algebras}


\author[Goetz ]{ }

\maketitle

\begin{center}

\vskip-.2in Peter Goetz \\
\bigskip

Department of Mathematics\\ Humboldt State University\\
Arcata, California  95521
\\ \ \\

\end{center}

\setcounter{page}{1}

\thispagestyle{empty}

\vspace{0.2in}

\begin{abstract}

Let $\k$ be a field, and let $A$ and $B$ be connected $\N$-graded $\k$-algebras. The algebra $A$ is said to be a graded right-free extension of $B$ provided there is a surjective graded algebra morphism $\pi: A \to B$ such that $\ker\pi$ is free as a right $A$-module. Suppose that $B$ is graded left coherent, and that $A$ is a graded right-free extension of $B$. We characterize when $A$ is also graded left coherent. We apply our criterion to prove graded coherence of certain non-Noetherian graded twisted tensor products.
\end{abstract}

\bigskip

\section{Introduction}
\label{introduction}

A graded associative algebra is \emph{graded left coherent} if every finitely generated graded left ideal is finitely presented. It is easy to prove that every graded left Noetherian algebra is graded left coherent. In some approaches to the subject of noncommutative projective algebraic geometry, for example Artin-Zhang \cite{Artin-Zhang}, graded Noetherian algebras take a prominent role. However, many naturally occurring graded algebras, while being non-Noetherian, happen to be graded coherent. Moreover, as we now explain, the projective geometry of graded coherent algebras can be studied in a manner similar to the theory developed in Verevkin \cite{Verevkin}, and Artin-Zhang \cite{Artin-Zhang}, for example. 

Let $\k$ be a field. Suppose $A$ is a connected $\N$-graded $\k$-algebra that is finitely generated in degree $1$. Let $\finpres(A)$ denote the category of graded finitely presented left $A$-modules. In general, $\finpres(A)$ is not an abelian category, and so $\finpres(A)$ is not amenable to study via standard homological algebra. However, if $A$ is graded left coherent, then $\finpres(A)$ is an abelian category. Let $\fdim(A)$ denote the full subcategory of $\finpres(A)$ consisting of the graded finite-dimensional left $A$-modules. The quotient category $$\qgr(A) = \finpres(A)/\fdim(A)$$ is also an abelian category. Then, guided by Serre's theorem on the equivalence of categories of coherent sheaves on projective schemes and certain module categories, one considers $\qgr(A)$ as a suitable replacement for the category of coherent sheaves on the nonexistent projective scheme associated to $A$. We refer to the papers of Verevkin \cite{Verevkin}, Artin-Zhang \cite{Artin-Zhang}, and Polishchuk \cite{Polishchuk} for more details and background. 

Given a graded $\k$-algebra $A$, it can be a challenging problem to determine if $A$ is graded coherent. Recently, in joint work with Conner, \cite{Conner-Goetz}, we classified the quadratic twisted tensor products of $\k[x,y]$ and $\k[z]$. Additionally, we characterized the algebras in this class that are graded left Noetherian. It is then a natural problem to determine which of the quadratic twisted tensor products of $\k[x,y]$ and $\k[z]$ are graded left (or right) coherent. Using a theorem of Piontkovskii, see Theorem \ref{Piontkovskii coherence} below, we have checked that many of the non-Noetherian algebras are indeed graded coherent. However, there are algebras in this class where it is not clear if the hypotheses of Theorem \ref{Piontkovskii coherence} hold. Thus we were motivated to extend Piontkovskii's result. 


Now we state and discuss our main results. The following definition is fundamental for us.

\begin{defn}
Let $A$ and $B$ be graded algebras. We say that $A$ is a \emph{graded right-free extension} of $B$ if there is a short exact sequence $$
\xymatrix{
0 \ar[r] & I \ar[r] & A \ar[r]^{\pi} & B \ar[r] & 0, }$$
where $\pi:A \to B$ is a graded algebra morphism, and $I = \ker\pi$ is a free right $A$-module. We also say that $A$ is a  \emph{graded right-free extension of $B$ by the ideal $I$}, in this situation.
\end{defn}

We defer all statements of other definitions to Section 2.


Piontkovskii \cite{Piontkovski} has proven the following useful theorem for checking graded coherence. For example, He-Oystaeyen-Zhang \cite{H-O-Z} have used Theorem \ref{Piontkovskii coherence} to prove graded coherence of certain Ore extensions of 2-Calabi-Yau algebras.

\begin{thm}[\cite{Piontkovski}, Proposition 3.2]
\label{Piontkovskii coherence}
Let $A$ and $B$ be graded algebras. Suppose that $A$ is a graded right-free extension of $B$. If $B$ is graded left Noetherian, then $A$ is graded left coherent. 
\end{thm}


One of the motivations for this work was to see if one could relax the condition in Theorem \ref{Piontkovskii coherence} that $B$ is graded left Noetherian to the condition: $B$ is graded left coherent. Unfortunately, without further assumptions, this is not the case. Our first main result is the following.

\begin{thm}
\label{counterexample to naive} There are graded algebras $A$ and $B$ such that $A$ is a graded right-free extension of $B$, the algebra $B$ is graded left coherent, and $A$ is not graded left coherent. 
\end{thm}

On a positive note we prove, through analysis of a change of rings spectral sequence, the following criterion.

\begin{thm}
\label{new criterion on coherence}
Let $A$ and $B$ be graded algebras such that $B$ is graded left coherent. Suppose that $A$ is a graded right-free extension of $B$ by the ideal $I$. Then:
\begin{itemize}
\item[(1)] if $J$ is a finitely generated graded left ideal of $A$, then $(I \cap J)/(IJ)$ is a finitely generated left $B$-module;
\item[(2)] $A$ is graded left coherent if and only if for every finitely generated graded left ideal $J$ of $A$ the left $B$-module $(I \cap J)/(IJ)$ is finitely presented. 
\end{itemize}
\end{thm}

Although condition (2) of Theorem 1.4 is admittedly technical we prove the following result as an instance where it is automatically satisfied.

\begin{thm}
\label{useful theorem}
Let $A$ and $B$ be graded algebras such that $B$ is graded left coherent. Suppose that $A$ is a graded right-free extension of $B$ by the ideal $I$; identify $B$ with $A/I$. Furthermore, assume that there is a vector space decomposition $B = C + D$, where $C$ is a graded left Noetherian subalgebra of $B$, and $D$ is a graded left ideal of $B$ with a finite homogeneous generating set $\{{\overline z}_i = z_i+I:  1 \leq i \leq m\}$ such that $z_i I = 0$ for all $1 \leq i \leq m$. Then $A$ is graded left coherent.  
\end{thm}

Here is the outline of the paper. In the preliminary Section \ref{preliminaries} we gather together the relevant definitions and basic results on coherence. Section \ref{main results} contains the proofs of our main theorems: Theorem \ref{counterexample to naive}, Theorem \ref{new criterion on coherence}, and Theorem \ref{useful theorem}. In Section \ref{examples} we show how Theorem \ref{useful theorem} can be used to prove graded coherence of certain graded twisted tensor products. We conclude the paper with some questions.

\section{Preliminaries}
\label{preliminaries}

Let $\k$ be a field. A connected $\N$-graded $\k$-algebra is a unital, associative $\k$-algebra, $A = \oplus_{n \geq 0} A_n$ such that $A_0 = \k$ and $A$  is generated by finitely many homogeneous elements. This implies that $A$ is \emph{locally finite}, i.e., that $\dim_{\k} A_n < \infty$ for all $n \geq 0$. In this paper, the term \emph{graded algebra} will refer, exclusively, to a connected $\N$-graded $\k$-algebra.  We will only consider left $A$-modules which are $\Z$-graded, so the term \emph{graded} $A$-\emph{module} will refer to a left $A$-module $M = \oplus_{n \in \Z} M_n$ such that $A_m M_n \subseteq M_{m+n}$. Given $d \in \Z$ and a graded module $M$, we write $M(d)$ for the module $M$ with shifted grading: $M(d)_n = M_{d+n}$.  We will write $A^n$ for the free $A$-module of rank $n$; we mainly work with finitely generated free modules, so usually $n \in \N$. A distinguished role is played by the \emph{trivial $A$-module} which is the $A$-module $\k = A/A_+$, where $A_+ = \oplus_{n > 0} A_n$. Tensor products taken with respect to $\k$ are denoted by $\tsr$.


For a graded algebra $A$, a graded right $A$-module $M$ and a graded left $A$-module $N$, we denote by $\Tor_n^A(M, N)$ the value of the $n$th left derived functor of the functor $M \tsr_A \underline{\hspace {.2cm}}$ on the module $N$. Thus, to determine $\Tor_n^A(M, N)$ one requires a projective resolution of $N$. As is well known, $\Tor_n^A(M, N)$ can also be computed from a projective resolution of $M$. The space $\Tor_n^A(M, N)$ inherits an internal grading coming from the usual grading on the tensor product: $$(M \tsr_A N)_n = \bigoplus_{l+m = n} M_l \tsr_A N_m.$$

\subsection{Graded coherence}

The following are the basic definitions we will need regarding the notion of coherence.

\begin{defn}
Let $A$ be a graded algebra. A graded $A$-module $M$ is \emph{graded coherent} if 
\begin{itemize}
\item[(1)] $M$ is finitely generated;
\item[(2)] the kernel of any graded morphism $A^n \to M$ (not necessarily surjective) is finitely generated.
\end{itemize}
Equivalently, $M$ is \emph{graded coherent} if $M$ is finitely generated and every finitely generated graded submodule of $M$ is finitely presented.
\end{defn}

Clearly, every graded coherent module is finitely presented. 

\begin{defn} Let $A$ be a graded algebra. We say that $A$ is \emph{graded left coherent} if the left regular module, $_A A$, is graded coherent.
\end{defn} 

It is easy to show that a graded algebra $A$ is graded left coherent if and only if every finitely generated graded left ideal of $A$ is finitely presented.

Next we state the basic results we will need about graded coherence. These results are certainly well known to experts, at least in the case of commutative, not necessarily graded, rings. A basic reference is Glaz' definitive text \cite{Glaz}. To help keep the paper self-contained, and for lack of a good reference in the case of connected graded $\k$-algebras we give proofs of a few of the results. The Snake Lemma will be very useful.

\begin{lemma}[\cite{Weibel}, Snake Lemma 1.3.2]
\label{snake}
Let $A$ be a graded algebra. Consider a commutative diagram of graded $A$-modules of the form 
$$\xymatrix{
& &A' \ar[d]^{f} \ar[r]^{q} & B' \ar[d]^{g} \ar[r]^{p} & C' \ar[d]^{h} \ar[r] & 0 \\
& 0 \ar[r] & A \ar[r]^{i} & B \ar[r]^{j} & C. &
}$$
If the rows are exact, then there is an exact sequence: 
$$\xymatrix{
&\ker f \ar[r]^{q} & \ker g \ar[r]^{p} & \ker h \ar[r]^{\d} & \coker f \ar[r]^{\overline{i}} & \coker g\ar[r]^{\overline{j}} & \coker h.
}$$

\end{lemma}

A nice consequence of the Snake Lemma is the next result.

\begin{lemma}[\cite{Weibel}, Exercise 3.2.5]
\label{Weibel exercise}
Let $A$ be a graded algebra; let $M$ be a finitely presented graded $A$-module. Let $F$ be a finitely generated graded $A$-module, and $\phi: F \to M$ a graded surjective morphism. Then $\ker \phi$ is finitely generated.
\end{lemma}

The next result is well known and a proof can easily be given from definitions so we omit a proof.
\begin{prop}
\label{finitely generated submodule of coherent module is coherent}
Every finitely generated graded submodule of a graded coherent module is graded coherent.
\end{prop}


A fundamental fact that we use in Section 3 is the following well-known result. We include a proof for the sake of keeping this paper self-contained.

\begin{prop}
\label{2 of 3}
Let $A$ be a graded algebra. Let $M_1, M_2, M_3$ be graded $A$-modules and suppose
$$\xymatrix{
& 0 \ar[r] & M_1 \ar[r]^{i} & M_2 \ar[r]^{j} & M_3 \ar[r] &0
}$$
is an exact sequence. If any two of $M_1, M_2, M_3$ are graded  coherent, then so is the third.
\end{prop}

\begin{proof}
Firstly, suppose that $M_2$ and $M_3$ are graded coherent. We have to show that $M_1$ is graded coherent. Since $i$ is injective, it suffices to show that $ i(M_1)$ is graded coherent. Now $i(M_1)$ is a submodule of $M_2$, so, by Proposition \ref{finitely generated submodule of coherent module is coherent}, it will suffice to prove that $i(M_1)$ is finitely generated. Since $M_2$ is finitely generated we may choose a surjection $\pi: A^{a} \to M_2$. Notice that the composite map $A^{a}  \to M_3$ is surjective;  moreover, $M_3$ is coherent, hence $M_3$ is finitely presented. Therefore, by Lemma \ref{Weibel exercise}, we know that the kernel, $K$, of $A^a \to M_3$ is finitely generated. It is straightforward to check that $\pi$ maps $K$ onto $i(M_1)$, so 
it follows that $i(M_1)$ is finitely generated, as desired.

Secondly, suppose that $M_1$ and $M_3$ are graded coherent. We must show that $M_2$ is graded coherent. By definition $M_1$ and $M_3$ are finitely generated, so it is easy to show that $M_2$ is finitely generated. Let $g: A^n \to M_2$ be any graded morphism. We must prove that $\ker g$ is finitely generated. Consider the commutative diagram 
$$\xymatrix{
& &0 \ar[d] \ar[r] & A^n \ar[d]^{g} \ar[r]^{{\rm id}} & A^n \ar[d]^{j \circ g} \ar[r] & 0 \\
& 0 \ar[r] & M_1 \ar[r]^{i} & M_2 \ar[r]^{j} & M_3.&
}$$
The rows are exact, so, by the Snake Lemma, there is an exact sequence: 
$$0 \to \ker g \to \ker(j \circ g) \to M_1.$$

Since $M_3$ is graded coherent we know, by definition, that $\ker(j \circ g)$ is finitely generated. Let $N = \im(\ker(j \circ g) \to M_1)$. Then $N$ is finitely generated. Using the fact that $M_1$ is graded coherent, we know, by Proposition \ref{finitely generated submodule of coherent module is coherent}, that $N$ is graded coherent. Now consider the exact sequence: $$0 \to \ker g \to \ker(j \circ g) \to N \to 0.$$ By the first paragraph of this proof we conclude that the kernel of the map $$\ker(j \circ g) \to N$$ is finitely generated. From the exact sequence we know that $\ker g$ is isomorphic to this kernel, whence, $\ker g$ is finitely generated. We conclude that $M_2$ is graded coherent.

Thirdly, suppose that $M_1$ and $M_2$ are graded coherent. We must prove that $M_3$ is graded coherent. Since $M_2$ is finitely generated it is clear that $M_3$ is finitely generated. Let $f: A^n \to M_3$ be an arbitrary graded morphism. We need to show that $\ker f$ is finitely generated. Using the fact that $A^n$ is projective, construct a graded morphism $g: A^n \to M_2$ such that 
$$\xymatrix{
& & A^n \ar[ld]_{g} \ar[d]^{f} & \\
&M_2 \ar[r]^{j} & M_3 \ar[r] & 0
}$$ 
commutes. Since $M_1$ is finitely generated, choose a graded surjection $\pi: A^m \to M_1$. Now consider the diagram
$$\xymatrix{
& 0 \ar[r] & A^m \ar[d]^{\pi} \ar[r] & A^m \oplus A^n \ar[d]^{(i \circ \pi) \oplus g} \ar[r] & A^n \ar[d]^{f} \ar[r] & 0 \\
& 0 \ar[r] & M_1 \ar[r]^{i} & M_2 \ar[r]^{j} & M_3 \ar[r] & 0,
}$$
where the maps in the top row are the canonical ones. So the top row is exact. It is clear by inspection that the diagram commutes. Hence, by the Snake Lemma, we have an exact sequence 
$$\ker((i \circ \pi) \oplus g) \to \ker f \to (\coker \pi = 0).$$
Since $M_2$ is graded coherent we know that $\ker((\i \circ \pi) \oplus g)$ is finitely generated. It follows that $\ker f$ is finitely generated, and therefore $M_3$ is graded coherent. 
\end{proof}

A careful reading of the first paragraph of the last proof shows that we have proved the following porism. We need this result in Section 3, so we separate it out here.

\begin{lemma}
\label{finitely generated kernel}
Let $A$ be a graded algebra. Suppose that $$0 \to M_1 \to M_2 \to M_3 \to 0$$ is an exact sequence of graded $A$-modules such that $M_3$ is graded coherent and $M_2$ is finitely generated. Then $M_1$ is finitely generated.
\end{lemma}


An easy consequence of Proposition \ref{2 of 3} is the following.

\begin{cor}
\label{finite sums are coherent}
Let $A$ be a graded algebra. If $M_1$, $M_2$ are graded coherent $A$-modules, then $M_1 \oplus M_2$ is graded coherent.
\end{cor}

Let us conclude this section by remarking that if $A$ is a graded left coherent algebra, then the category of graded left coherent $A$-modules coincides with the category of graded finitely presented $A$-modules. Additionally, as follows immediately from Proposition \ref{2 of 3}, this category is an abelian category.
\section{Main Theorems}
\label{main results}

In this section we will give proofs of our main results. One of the motivations of this paper was to see if one could relax a hypothesis in Piontkovskii's theorem, Theorem \ref{Piontkovskii coherence}. Unfortunately, this is not the case, as we now explain.

\begin{defn}
Let $C = \k \la x, y, z \ra/ \la yz-zy, xz \ra$, considered as a connected $\N$-graded $\k$-algebra with $\deg(x) = \deg(y) = \deg(z) = 1$. Let $I$ be the two-sided ideal of $C$ generated by $z$.
\end{defn}

Consider the total order on the set of monomials in the free algebra $\k \la x, y, z \ra$ determined by $x < z < y$ and left-lexicographic order. With this order, in the terminology of \cite{Bergman}, there are no ambiguities to resolve. Hence, by \cite{Bergman} Theorem 1.2, $C$ has a monomial basis consisting of all monomials not containing $yz$ and $xz$. Denote this basis by $\mathscr{B}$.

We claim that $C$ is not graded left coherent. To see this, consider the left ideal of $C$ generated by $z$. Observe that for any $i \geq 0$ we have $xy^iz = xzy^i = 0.$ Moreover, considering the basis $\mathscr{B}$. makes it clear that the left annihilator ideal of $z$ in $C$ is generated by $\{xy^i : i \geq 0\}$, and that this ideal is not finitely generated as a left ideal of $C$. Hence $C$ is not graded left coherent.

Next, it is easy to show that $I = zC$. We claim that $I$ is free as a right $C$-module. To see this it suffices to show that the right annihilator ideal of $z$ in $C$ is the zero ideal. One easily checks that this is the case using the basis $\mathscr{B}$. Lastly, it is also easy to check that as algebras $C/I \cong \k \la x, y \ra$. By Corollary 3.2 of \cite{Polishchuk}, $\k \la x, y \ra$ is graded left coherent. Hence we have proved the following result. 

\begin{thm}
\label{counterexample to naive} There are graded algebras $A$ and $B$ such that $A$ is a graded right-free extension of $B$, the algebra $B$ is graded left coherent, and $A$ is not graded left coherent.
\end{thm}

Throughout the rest of this section let us fix the following notation. Let $A$ and $B$ be graded $\k$-algebras such that $A$ is a graded right-free extension of $B$ by the ideal $I$. There is no harm in identifying $B$ with $A/I$, so let us do so. Furthermore, we assume that $B$ is graded left coherent. Let $J$ denote a finitely generated graded left ideal of $A$. 

Our proof of Theorem \ref{new criterion on coherence} is based on an analysis of a certain spectral sequence, and we need some detailed homological information about two particular $B$-modules. Therefore we begin with two technical results.

\begin{lemma}
\label{Tor_0 has a good resolution}
The left $B$-module $\Tor_0^A(B, A/J)$ has a projective resolution consisting entirely of finitely generated free left $B$-modules.
\end{lemma}

\begin{proof}
We begin by identifying $\Tor_0^A(B, A/J) = B \tsr_A (A/J)$ as left $B$-modules. For $b \in B$, $a \in A$ we use the notation $b.a$ for the right action of $A$ on $B$. It is well known and easy to prove that there is a canonical isomorphism of left $B$-modules $$B \tsr_A (A/J) \cong B/(B.J).$$ We note for use below that $B/(B.J)$ is a cyclic left $B$-module generated by $1_B+B.J$. 

Next, we claim that $B.J$ is a finitely generated graded left ideal of $B$. It is obvious that $B.J$ is a graded left ideal of $B$. To prove $B.J$ is finitely generated as a $B$-module, let $x_1, \ldots, x_n$ be homogeneous generators of $J$ as a left $A$-module. Observe that $$x_i+I = (1_A+I)(x_i+I) = 1_B.x_i \in B.J.$$ Let $b \in B$ and let $y \in J$. Write $y = \sum a_i x_i$ for some $a_i \in A$. Then $$b.y = b.\sum a_i x_i = \sum(b.a_i).x_i = \sum (b.a_i)(x_i+I).$$ Hence $\{x_i+I: 1 \leq i \leq n\}$ generates $B.J$ as a left $B$-module.

We have shown that $B.J$ is a finitely generated submodule of $B$, and $B$ is graded left coherent as a $B$-module. Hence, by Proposition \ref{finitely generated submodule of coherent module is coherent}, $B.J$ is a graded left coherent $B$-module.
Consideration of the exact sequence of left $B$-modules $$0 \to B.J \to B \to B/B.J \to 0$$ and Proposition \ref{2 of 3} makes it clear that $B/B.J$ is a graded left coherent $B$-module. Then, by repeated use of Corollary \ref{finite sums are coherent}, it follows that $B/B.J$ has a projective resolution consisting entirely of finitely generated free left $B$-modules.
\end{proof}

\begin{lemma}
\label{Tor_1 has good properties}
The left $B$-module $\Tor_1^A(B, A/J)$ is finitely generated. Moreover, there is a canonical isomorphism of left $B$-modules $$\Tor_1^A(B, A/J) \cong (I \cap J)/(IJ).$$
\end{lemma}

\begin{proof}
Let $K = \Tor_1^A(B, A/J)$. Consider the short exact sequence of $A$-modules $$0 \to J \to A \to A/J \to 0.$$ Applying the functor $B \tsr_A \underline{\hspace{.25cm}}$ and taking homology yields the exact sequence $$0 \to K \to B \tsr_A J \to B \tsr_A A \to B \tsr_A A/J \to 0.$$ In particular, $K = \ker(1 \tsr i: B \tsr_A J \to B \tsr_A A)$, where $i: J \to A$ is the inclusion map. Hence $$0 \to K \to B \tsr_A J \to \im(1 \tsr i) \to 0$$ is exact. Since $J$ is finitely generated as a left $A$-module, it is easy to prove that $B \tsr_A J$ is finitely generated as a left $B$-module, and therefore $\im(1 \tsr i)$ is a finitely generated $B$-submodule of $B \tsr_A A$. By assumption $B$ is graded left coherent, so $B \tsr_A A \cong B$ is graded coherent as a left $B$-module. Hence, by Proposition \ref{finitely generated submodule of coherent module is coherent}, we know $\im(1 \tsr i)$ is graded coherent. Thus, using Lemma \ref{finitely generated kernel}, it follows that $K$ is a finitely generated $B$-module, as desired. 

For the second statement, there are canonical isomorphisms $$B \tsr_A J  = (A/I) \tsr_A J \cong J/(IJ), \quad B \tsr_A A \cong B = A/I.$$ Under these isomorphisms the map $1 \tsr i: B \tsr_A J \to B \tsr_A A$ is identified with the map $$\pi: J/(IJ) \to A/I, \quad x + IJ \mapsto x+I \text{ for } x \in J.$$ It is clear that $x + IJ \in \ker \pi$ if and only if $x \in I \cap J$. Hence we have $$\Tor_1^A(B, A/J) \cong (I \cap J)/(IJ).$$ Let us note that the left $B$-module structure on $(I \cap J)/(IJ)$ is given by $$(a+I).(y+IJ) = ay + IJ \quad \text{ for } a \in A, y \in I \cap J.$$
\end{proof}

We are now prepared to prove our main result. 
\begin{thm}
\label{new criterion on coherence}
Let $A$ and $B$ be connected graded $\k$-algebras such that $B$ is graded left coherent. Suppose that $A$ is a graded right-free extension of $B$ by the ideal $I$. Then:
\begin{itemize}
\item[(1)] if $J$ is a finitely generated graded left ideal of $A$, then $(I \cap J)/(IJ)$ is a finitely generated left $B$-module;
\item[(2)] $A$ is graded left coherent if and only if for every finitely generated graded left ideal $J$ of $A$ the left $B$-module $(I \cap J)/(IJ)$ is finitely presented. 
\end{itemize}
\end{thm}

\begin{proof}
Statement (1) follows immediately from Lemma \ref{Tor_1 has good properties}.

Let $J$ be a finitely generated graded left ideal of $A$. Consider the first quadrant homology change-of-rings spectral sequence (see \cite{Weibel} Theorem 5.6.6 for example) $$E^2_{p, q} = \Tor_p^B(\k, \Tor_q^A(B, A/J)) \implies \Tor_{\ast}^A(\k, A/J).$$ 

First, assume that the left $B$-module $(I \cap J)/(IJ)$ is finitely presented. We must prove that $J$ is finitely presented as a left $A$-module. Equivalently, we must show that $$\dim_{\k} \Tor_2^A(\k, A/J) < \infty.$$

Since $I$ is a free right $A$-module, $$0 \to I \to A \to B \to 0$$ is a graded free resolution of $B$ as a right $A$-module. Therefore $\Tor_q^A(B, A/J) = 0$ for all $q \geq 2$. It follows from the spectral sequence that $$\dim_{\k} \Tor_2^A(\k, A/J) \leq \dim_{\k} E^2_{2, 0} + \dim_{\k} E^2_{1, 1}.$$ Now, Lemma \ref{Tor_0 has a good resolution} implies that $\dim_{\k} E^2_{2, 0} < \infty$. Lemma \ref{Tor_1 has good properties} and our assumption implies that $\dim_{\k} E^2_{1, 1} < \infty$. Thus, $\dim_{\k} \Tor_2^A(\k, A/J) < \infty.$ We conclude that $A$ is graded left coherent.

Finally, suppose that $J$ is a finitely generated graded left ideal of $A$ such that the left $B$-module $(I \cap J)/(IJ)$ is not finitely presented. We know, by statement (1), that the left $B$-module $(I \cap J)/(IJ)$ is finitely generated. So there is an exact sequence of graded $B$-modules $$F_1 \to F_0 \to (I \cap J)/(IJ) \to 0 \quad (\ast),$$ where $F_0$ is a finitely generated free module, and $F_1$ is a non-finitely generated free module. 

Recall that $\Tor_q^A(B, A/J) = 0$ for all $q \geq 2$, so one easily proves that, as graded $\k$-vector spaces, $$\Tor_2^A(\k, A/J) \cong E^3_{2,0} \oplus E^3_{1,1}.$$ We claim that $E^3_{1,1}$ is infinite dimensional as a $\k$-vector space. To see this, note that $E^3_{1,1}$ is the homology of $$\xymatrix{0 \ar[r] & E^2_{3,0} \ar[r]^{d^2_{3,0}} & E^2_{1,1} \ar[r] &0}$$
at the $E^2_{1,1}$ position. The exact sequence $(\ast)$ and Lemma \ref{Tor_1 has good properties} shows that $$E^2_{1,1} = \Tor_1^B(\k,\Tor_1^A(B, A/J))$$ is infinite dimensional as a $\k$-vector space. Furthermore, Lemma \ref{Tor_0 has a good resolution} implies, in particular, that $$E^2_{3,0} = \Tor^B_3(\k, \Tor^A_0(B, A/J))$$ is a finite-dimensional $\k$-vector space. Hence, we see that $E^3_{1,1}$ is an infinite-dimensional $\k$-vector space. Therefore, $\Tor_2^A(\k, A/J)$ is an infinite-dimensional $\k$-vector space, and so $J$ is not finitely presented as a left $A$-module. We conclude that $A$ is not graded left coherent.

\end{proof}

There are cases where the condition in statement (2) of Theorem \ref{new criterion on coherence} is automatically satisfied, as illustrated in our final main result. 

\begin{thm}
\label{useful theorem}
Let $A$ and $B$ be graded algebras such that $B$ is graded left coherent. Suppose that $A$ is a graded right-free extension of $B$ by the ideal $I$; identify $B$ with $A/I$. Furthermore, assume that there is a vector space decomposition $B = C + D$, where $C$ is a graded left Noetherian subalgebra of $B$, and $D$ is a graded left ideal of $B$ with a finite homogeneous generating set $\{{\overline z}_i = z_i+I:  1 \leq i \leq m\}$ such that $z_i I = 0$ for all $1 \leq i \leq m$. Then $A$ is graded left coherent.  
\end{thm}

\begin{proof}
Let $J$ be a finitely generated graded left ideal of $A$. Let $$M = (I \cap J)/(IJ).$$ By Theorem \ref{new criterion on coherence} (2) it suffices to prove that $M$ is finitely presented as a left $B$-module. Using Theorem \ref{new criterion on coherence} (1) let $$\{m_i \in M : 1 \leq i \leq n\}$$ be a finite homogeneous generating set for $_BM$. Let $d_i = \deg(m_i)$ for $1 \leq i \leq n$. Consider the canonical surjection of graded left $B$-modules $$\bigoplus_{i=1}^n B(-d_i) \to M, \quad 1_{B(-d_i)} \mapsto m_i;$$ and let $K = \ker(\oplus B(-d_i) \to M)$. We must prove that $K$ is a finitely generated left $B$-module.

The assumption that $z_i I = 0$ for all $1 \leq i \leq m$ ensures that ${\overline z}_i M = 0$, and so $DM = 0$. Let us consider $M$ as a left $C$-module by restricting the action of $B$ to its subalgebra $C$; denote this module by $_C M$. We claim that $_C M$ is finitely generated. To see this note that $$M= \sum_{i=1}^n B m_i = \sum_{i=1}^n(C+D)M = \sum_{i=1}^n Cm_i,$$ as desired.

Next, consider the following surjection of graded left $C$-modules $$\bigoplus_{i=1}^n C(-d_i) \to M, \quad 1_{C(-d_i)} \mapsto m_i;$$ and let $L = \ker(\oplus C(-d_i) \to M)$. Since $C$ is graded left Noetherian we know that $L$ is finitely generated as a left $C$-module. Let $\{{\overline g}_k : 1 \leq k \leq l\}$ be a set of homogeneous generators of $_CL$. Let us also define $${\overline z}_{ji} = (0, \ldots, {\overline z}_j, \ldots, 0) \in \bigoplus_{i=1}^n B(-d_i) \text{ for all } 1 \leq j \leq m, \, 1 \leq i \leq n,$$ where ${\overline z}_j$ is in the $i$th position. Let $$S = \{{\overline z}_{ji} : 1 \leq j \leq m, 1 \leq i \leq n\} \cup \{{\overline g}_k : 1 \leq k \leq l\}.$$ 

We claim that $K$ is generated as a left $B$-module by the set $S$. First, it is clear that $S \subset K$. Let $K'$ denote the $B$-submodule of $K$ generated by $S$. Let $(b_1, \ldots, b_n) \in K$. Write, for each $i$, $1 \leq i \leq n$, 
$$b_i = c_i + \sum_{j=1}^m b'_{ij} {\overline z}_j  \text{ for some } c_i \in C,\, b'_{ij} \in B.$$ Then we have 
$$0 = \sum_{i=1}^n b_i m_i = \sum_{i=1}^n (c_i + \sum_{j=1} ^m  b'_{ij} {\overline z}_j) m_i = \sum_{i=1}^n c_i m_i.$$ Therefore $(c_1, \ldots, c_n) \in L$ and we may write $$(c_1, \ldots, c_n) = \sum_{k=1}^l e_k {\overline g}_k \text{ for some } e_k \in C.$$ Now observe that 
\begin{align*}
(b_1 , \ldots, b_n) &= (c_1 + \sum_{j=1}^m b'_{1 j} {\overline z}_j, \ldots, c_n + \sum_{j=1}^m b'_{n j} {\overline z}_j) \\
&= (c_1, \ldots, c_n) + \sum_{j=1}^m b'_{1 j} {\overline z}_{j 1} + \cdots + \sum_{j=1}^m b'_{n j} {\overline z}_{j n} \\
&= \sum_{k=1}^l e_k {\overline g}_k + \sum_{j=1}^m b'_{1 j} {\overline z}_{j 1} + \cdots + \sum_{j=1}^m b'_{n j} {\overline z}_{j n}.
\end{align*}

Hence we see that $(b_1, \ldots, b_n) \in K'$, and so $K$ is finitely generated as a left $B$-module, as claimed.

We conclude that $M$ is a finitely presented left $B$-module, and so, by Theorem \ref{new criterion on coherence} (2), $A$ is graded left coherent.

\end{proof}

\section{Coherence of certain twisted tensor products}
\label{examples}

One of the motivations for this paper was to determine if certain twisted tensor products constructed in \cite{Conner-Goetz} are graded coherent. We briefly review the notion of a graded twisted tensor product. Let $A$ and $B$ be $\N$-graded $\k$-algebras. Let $\mu_A: A \tsr A \to A$ and $\mu_B: B \tsr B \to B$ denote the multiplication maps of $A$ and $B$, respectively. Endow the $\k$-linear tensor product $A \tsr B$ with an $\N$-grading via $$(A \tsr B)_m = \bigoplus_{k+l = m} A_k \tsr B_l.$$ Let $\t: B \tsr A \to A \tsr B$ be an $\N$-graded $\k$-linear map. We say that $\t$ is a \emph{twisting map} if $\t(1 \tsr a) = a \tsr 1$ and $\t(b \tsr 1) = 1 \tsr b$ for all $a \in A$ and $b \in B$, and $$\t(\mu_B \tsr \mu_A) = (\mu_A \tsr \mu_B)(1 \tsr \t \tsr 1)(\t \tsr \t)(1 \tsr \t \tsr 1).$$ Given a twisting map $\t: B \tsr A \to A \tsr B$, the space $A \tsr B$ carries the structure of a unital associative algebra with multiplication given by $\mu_{\t}: A \tsr B \tsr A \tsr B \to A \tsr B$, where $\mu_{\t} = (\mu_A \tsr \mu_B)(1 \tsr \t \tsr 1)$. We refer to $(A \tsr B, \mu_{\t})$ as \emph{the twisted tensor product of $A$ and $B$ associated to $\t$}. We also write $A \tsr_{\t} B$ for this algebra. For more details on graded twisted tensor products we refer the reader to \cite{Conner-Goetz1}, for example. 

Theorem \ref{useful theorem} enables us to prove the following result about certain graded twisted tensor products.

\begin{thm}
\label{zero product ttps}
Let $A$ be a quadratic twisted tensor product of $\k[x.y]$ and $\k[z]$ associated to a graded twisting map $\t: \k[z] \tsr \k[x,y] \to \k[x,y] \tsr \k[z]$. Suppose that $\t(z \tsr x) = 0$. Then $A$ is graded left coherent.
\end{thm}

\begin{proof}
We check that the criteria of Theorem \ref{useful theorem} are satisfied. Let $I = xA$ be the right ideal of $A$ generated by $x$. The assumption that $\t(z \tsr x) = 0$ makes it easy to see that $I$ is a two-sided ideal of $A$. Moreover, the fact that $A$ is a twisted tensor product of the form $\k[x,y] \tsr_{\t} \k[z]$ ensures that $I$ is a free right $A$-module. Let $B = A/I$. Then we have proved that $A$ is a graded right-free extension of $B$ by the ideal $I$. 

Next, we note that it is easy to prove that $$B \cong \k \la y, z \ra/\la zy - \a y^2 - \b yz - \g z^2 \ra,$$ for some $\a, \b, \g \in \k$. Then, using \cite{Piontkovski} Theorem 1.2, we know that $B$ is graded left coherent. Moreover, one easily checks that the conditions of \cite{Conner-Goetz1} Theorem 3.7 are satisfied for the algebra $B$. Therefore $$B \cong \k[y] \tsr_{\s} \k[z],$$ where $\s: \k[z] \tsr \k[y] \to \k[y] \tsr \k[z]$ is the graded twisting map determined by $\s(z\tsr y) = \a y^2 \tsr 1+\b y \tsr z + 1 \tsr \g z^2$. 

Let $C$ be the subalgebra of $B$ generated by the element $y+I$. Let $D$ denote the left ideal of $B$ generated by the element $z + I$. Then the fact that $B \cong \k[y] \tsr_{\s} \k[z]$ makes it clear that there is a $\k$-vector space decomposition $$B = C \oplus D.$$ Additionally, it is clear that $C$ is isomorphic to the polynomial algebra $\k[y]$. In particular, $C$ is graded left Noetherian. 

We conclude, by Theorem \ref{useful theorem}, that $A$ is graded left coherent. 
\end{proof}

\begin{ex} Let $A = \k \la x, y, z \ra/ \la xy-yx, zx, zy \ra.$ Then $A$ is a quadratic twisted tensor product of $\k[x,y]$ and $\k[z]$. By Theorem \ref{zero product ttps}, $A$ is graded left coherent. We remark that the quotient algebra $B = A/I$, where $I = xA$, is not graded left Noetherian. So, it is not clear if one can apply Piontkovskii's result: Theorem \ref{Piontkovskii coherence}, to prove the graded left coherence of $A$.

\end{ex}

We conclude this paper with some questions. 

\begin{ques}
In \cite{Piontkovski1} Piontkovski has proved that monomial algebras are graded coherent. Can one prove this result using the criterion in Theorem \ref{new criterion on coherence}?
\end{ques}

It is well known that the analogue of the Hilbert basis theorem is false for coherent rings. Namely, Soublin \cite{Soublin} has constructed a coherent ring $R$ such that the polynomial ring $R[x]$ is not coherent. In contrast, recently Minamoto \cite{Minamoto} has proved that if $B$ is a coherent algebra, then the polynomial algebra $B[x]$ is graded coherent, where the grading on $B[x]$ is given by placing $B$ in degree $0$ and $\deg(x) = 1$. An analogous result in the context of connected $\N$-graded $\k$-algebras would be interesting.

\begin{ques}
\label{polynomial coherence}
Let $B$ be a graded left coherent algebra. Consider the polynomial ring $B[x]$ as a connected $\N$-graded $\k$-algebra with $\deg(x) = 1$. Must $B[x]$ be graded left coherent?
\end{ques}

One might attempt to answer this question by applying the criterion of Theorem \ref{new criterion on coherence} to the two-sided ideal of $B[x]$ generated by $x$. We remark in passing that we have checked that the analogue of Question \ref{polynomial coherence} has a positive answer in the case of the truncated polynomial ring $B[x]/\la x^n \ra$.

%

\bibliographystyle{plain}
\bibliography{bibliog}

\end{document}